\documentclass[reqno, 12pt,draft]{amsart}
\usepackage{amsmath,amsthm,amsfonts,amssymb}

\date{}
\newtheorem{theorem}{Theorem}[section]
\newtheorem{lemma}[theorem]{Lemma}
\newtheorem{corollary}[theorem]{Corollary}
\newtheorem{remark}[theorem]{Remark}
\newtheorem{proposition}[theorem]{Proposition}
\newtheorem{example}[theorem]{Example}
\numberwithin{equation}{section}

\begin{document}

\centerline{{\bf On the Ambrosio--Figalli--Trevisan superposition principle}}

\vskip .1cm 
\centerline{{\bf for probability solutions to Fokker--Planck--Kolmogorov equations}}

\vskip .5cm

\centerline{\bf Vladimir I. Bogachev$^{a}$\footnote{corresponding author.
{\it E-mail addresses}: vibogach@mail.ru (V. Bogachev),
roeckner@math.uni-bielefeld.de (M. R\"ockner), starticle@mail.ru (S. Shaposhnikov)},
Michael R\"ockner$^{b}$,
Stanislav V. Shaposhnikov$^{a}$}

\vspace*{0.2cm}

\noindent
\small{$^{a}${\it Lomonosov Moscow State University, Russia and National Research
University Higher School of Economics, Russian Federation}

\noindent
$^{b}${\it Fakult\"at f\"ur Mathematik, Universit\"at Bielefeld,
D-33501 Bielefeld, Germany}
}

\vspace*{0.3cm}

{\bf Abstract} We prove a generalization of the known result of Trevisan on the
Ambrosio--Figalli--Trevisan superposition principle for probability solutions to the Cauchy problem
for the Fokker--Planck--Kolmogorov equation, according to which such a solution is generated
by a solution to the corresponding martingale problem.
The novelty is that in place of the integrability of the diffusion and drift coefficients
$A$ and $b$ with respect to the solution we require the integrability
of $(\|A(t,x)\|+|\langle b(t,x),x\rangle |)/(1+|x|^2)$. Therefore, in the case where there are no
a~priori global integrability conditions the function $\|A(t,x)\|+|\langle b(t,x),x\rangle |$
can be of quadratic growth. Moreover, as a corollary we obtain that under mild conditions on the initial distribution
it is sufficient to have the one-sided bound $\langle b(t,x),x\rangle \le C+C|x|^2 \log |x|$
along with $\|A(t,x)\|\le C+C|x|^2 \log |x|$.

\vskip .1in

Keywords: Fokker--Planck--Kolmogorov equation, martingale problem, superposition principle

\vskip .1in

AMS MSC 2010: 60J60; 35Q84

\section{Introduction}

We study solutions to the  Cauchy problem for the Fokker--Planck--Kolmogorov equation
\begin{equation}\label{z1}
\partial_t\mu_t=\partial_{x_i}\partial_{x_j}\bigl(a^{ij}\mu_t\bigr)
-\partial_{x_i}\bigl(b^i\mu_t\bigr), \quad \mu_0=\nu .
\end{equation}
Below we write this equation in the short form $\partial_t\mu_t=L^{*}\mu_t$,
where $L^{*}$ is the formal adjoint operator to the differential operator
$$
Lu=a^{ij}\partial_{x_i}\partial_{x_j}u+b^i\partial_{x_i}u,
$$
where the usual convention about summation over repeated indices is employed.
We assume throughout that the matrix  $A(t,x)=(a^{ij}(t,x))_{i,j\le d}$ is symmetric and
nonnegative definite and the functions $(t,x)\mapsto a^{ij}(t,x)$
and $(t,x)\mapsto b^i(t,x)$ are Borel measurable on $[0,T]\times\mathbb{R}^d$.
By a solution we
mean a mapping $t\mapsto \mu_t$ from $[0, T]$ to the space
of probability measures $\mathcal{P}(\mathbb{R}^d)$ that is
continuous with respect to the weak topology and satisfies the integral equality
$$
\int_{\mathbb{R}^d}\varphi\,d\mu_t=\int_{\mathbb{R}^d}\varphi\,d\nu
+\int_0^t\int_{\mathbb{R}^d} L\varphi\,d\mu_s\,ds
$$
for all $t\in[0, T]$ and all $\varphi\in C_0^{\infty}(\mathbb{R}^d)$, where it is assumed
that $a^{ij}$ and $b^i$ are locally (i.e., on compact sets
in $[0,T]\times \mathbb{R}^d$)
integrable with respect to the measure~$\mu_t\, dt$:
$$
a^{ij}, b^i\in L^1_{loc} (\mu_t\, dt).
$$
The measure $\mu=\mu_t\, dt$ on $[0,T]\times\mathbb{R}^d$ is defined as usual by the equality
$$
\int f\, d\mu=\int_0^T \int_{\mathbb{R}^d} f(t,x)\, \mu_t(dx)\, dt
$$
This measure can be identified with the solution, which is also denoted by $\{\mu_t\}$.

Recall that the weak topology on $\mathcal{P}(\mathbb{R}^d)$ is generated by the seminorms
$$
\sigma\mapsto \biggl|\int_{\mathbb{R}^d} f(x)\, \sigma(dx)\biggr|
$$
on the linear space of all bounded Borel measures, where $f$ is a bounded continuous function
on~$\mathbb{R}^d$ (see~\cite{B18book}).
Recent accounts on the theory of Fokker--Planck--Kolmogorov equations can be found
in \cite{BKR-umn} and \cite{book}.
The inner product and norm on $\mathbb{R}^d$ are denoted by $\langle x,y\rangle$ and~$|x|$,
respectively. The operator norm of a matrix $A$ is denoted by~$\|A\|$.

In the case $A=0$ (the continuity equation) the following
superposition principle
of Ambrosio \cite{A2008} is known  (see also~\cite{A17},
\cite{AT}, and~\cite{ST}). If $\mu=\mu_t\, dt$ with probability
measures~$\mu_t$ on~$\mathbb{R}^d$ satisfies  the continuity equation
$$
\partial_t\mu_t+{\rm div}(b\mu_t)=0
$$
 and
$|b(x,t)|/(1+|x|)$ is $\mu$-integrable, then there exists a nonnegative
Borel measure $\eta$ on the space $\mathbb{R}^d\times C([0,T],\mathbb{R}^d)$
concentrated on the set of pairs $(x,\omega)$ such
that $\omega$ is an absolutely continuous solution of the integral
equation
$$
\omega(t)=x+\int_0^t b(\omega(s),s)\, ds
$$
and, for each function $\varphi\in C_b(\mathbb{R}^d)$ and each $t\in [0,T]$, one has
$$
\int \varphi(x)\, \mu_t(dx)=\int \varphi(\omega(t))\, \eta (dx d\omega).
$$
In other words, the measure $\mu_t$ coincides with the image of $\eta$
under the evaluation mapping $(x,\omega)\mapsto \omega(t)$.
Of course, the integral on the right coincides with the integral against the projection
of $\eta$ on~$C([0,T],\mathbb{R}^d)$ (see \cite[Remark~3.1]{A2008} about the connection
between the two measures). A discussion of analogous representations for signed solutions
and interesting counter-examples can be found in~\cite{GM}.

The case of possibly nonzero $A$ and bounded~$A$ and~$b$
was first considered by Figalli~\cite{F}, who proved that every probability
solution to the Cauchy problem
for the Fokker--Planck--Kolmgorov equation is represented by a martingale measure
on the path space.
Generalizing this seminal achievement,
Trevisan  \cite{TREV} obtained the following important and very general result.

Suppose that a mapping $t\mapsto \mu_t$ from $[0, T]$ to the space
of probability measures $\mathcal{P}(\mathbb{R}^d)$ is  continuous with respect to the weak topology
and satisfies the  Cauchy problem (\ref{z1}).
Suppose also that it satisfies the condition
\begin{equation}\label{con-tr}
\int_0^T\int_{\mathbb{R}^d}\bigl[\|A(t, x)\|+|b(t, x)|\bigr]
\mu_t(dx)\,dt<\infty.
\end{equation}
Then there exists a Borel probability measure $P_{\nu}$ on the path space
$$
\Omega_d:=C([0, T],\mathbb{R}^d)
$$
 of continuous functions $\omega\colon [0,T]\to \mathbb{R}^d$
with its standard sup-norm $\|\omega\|=\sup_t |\omega(t)|$  such  that

\, (i) \, $P_{\nu}\bigl(\omega\colon \omega(0)\in B\bigr)=\nu(B)$ for all Borel
sets $B\subset\mathbb{R}^d$,

\, (ii) \, for every function $f\in C_0^{\infty}(\mathbb{R}^d)$, the function
 $$
 (\omega, t)\mapsto f(\omega(t))-f(\omega(0))-\int_0^tLf(s,\omega(s))\,ds
 $$
is a martingale with respect to the measure $P_{\nu}$ and the natural filtration
$\mathcal{F}_t=\sigma(\omega(s), s\in[0, t])$,

\, (iii) \, for every function $f\in C_0^{\infty}(\mathbb{R}^d)$, there holds the equality
$$
\int_{\mathbb{R}^d} f\,d\mu_t=\int_{\Omega_d}f(\omega(t))\,P_{\nu}(d\omega)
\quad \forall\, t\in [0,T].
$$

The latter means that $\mu_t$ is the law of $\omega(t)$ under~$P_\nu$, while (i) means
that $\nu$ is the law of~$\omega(0)$.

In spite of a  very general character of condition (\ref{con-tr}), in  many
simple situations it is not fulfilled. Let us consider a one-dimensional example. Let
$$
\varrho\in C^{\infty}(\mathbb{R}), \quad \varrho>0, \quad
\int\varrho(x)\,dx=1, \quad b(x)=\varrho'(x)/\varrho(x).
$$
Then $\mu_t(dx)=\mu(dx)=\varrho\,dx$ is a stationary solution to the
Fokker--Planck--Kol\-mo\-go\-rov equation
$$\partial_t\mu=\mu''-(b\mu)'.
$$
In particular, $\mu_t=\mu$ satisfies the  Cauchy problem with initial data $\mu$.
However, it is easy to find a smooth probability density $\varrho$ such that
$$
\int_{\mathbb{R}}|b(x)|\varrho(x)\,dx
=\int_{\mathbb{R}}|\varrho'(x)|\,dx=\infty.
$$

In this paper we reinforce the aforementioned result by replacing condition (\ref{con-tr})
with a weaker assumption.

Throughout  we assume  that the coefficients
are Borel measurable on $[0, T]\times\mathbb{R}^d$,
$$
a^{ij}, b^i\in L^1([0, T]\times U, \mu_t\,dt)
$$
for every ball~$U$ in~$\mathbb{R}^d$,
and the following condition is fulfilled:
\begin{equation}\label{con-new}
\int_0^{T}\int_{\mathbb{R}^d}
\frac{\|A(t, x)\|+|\langle b(t, x), x\rangle|}{(1+|x|)^{2}}
\,\mu_t(dx)\,dt<\infty.
\end{equation}

It follows from Proposition \ref{pr1} below (see also Example~\ref{ex1}) that in order to ensure
condition (\ref{con-new}) it suffices that $\log(1+|x|)\in L^1(\nu)$ and
$$
\|A(t, x)\|\le C+C|x|^2\log(1+|x|), \quad \langle b(t, x), x\rangle
\le C+C|x|^2\log(1+|x|).
$$
Obviously, it is also sufficient without any assumptions about $\nu$ that
$$
\|A(t, x)\|+|\langle b(t, x), x\rangle|\le C+C|x|^2.
$$

Our main result is the following theorem (its proof is given in the last section).

\begin{theorem}\label{tmain}
Suppose that $\{\mu_t\}$ is a solution to the  Cauchy problem $\partial_t\mu_t=L^{*}\mu_t$
on $[0, T]$ with $\mu_0=\nu$ and  {\rm (\ref{con-new})} is fulfilled. Then there exists a Borel probability
measure $P_{\nu}$ on $\Omega_d=C([0, T], \mathbb{R}^d)$ for which all assertions
{\rm (i)}, {\rm (ii)} and {\rm (iii)} are true.
\end{theorem}

It is important that our theorem assumes no uniqueness of probability solutions to the Cauchy problem,
the martingale representation exists for each probability  solution
satisfying~(\ref{con-new}).

It should be noted that the superposition principle does not work without
global assumptions even for smooth coefficients and $A=I$, because it can happen
that there are many probability solutions to the Fokker--Planck--Kolmogorov equation
(see \cite[Section~9.2]{book},
while the martingale  problem has a unique solution in this case
(see \cite[Corollary~10.1.2]{StrV}) and this solution necessarily
corresponds to  a subprobability solution to the FPK equation due to a blow up.

Note that the integrability of $(1+|x|)^{-2}|\langle b(t, x), x\rangle|$
can hold even in the case where the
function $(1+|x|)^{-1}|b(t, x)|$ is not integrable  with respect to the solution
(see Example~\ref{ex2}).

Not only is the assumption of integrability of $(1+|x|)^{-2}|\langle b(t, x),
x\rangle|$ weaker than the
assumption of integrability of $(1+|x|)^{-1}|b(t,x)|$, but it is also simpler to verify.
For example, as already noted above, if $\log (1+|x|)$ is $\nu$-integrable,
it suffices to have a one-sided
bound.

\begin{corollary}
Let $\log(1+|x|^2)\in L^1(\nu)$ and
$$
\|A(t,x)\|\le C+C|x|^2\log(1+|x|^2),
\quad
\langle b(t,x), x\rangle \le C+C|x|^2\log(1+|x|^2).
$$
Then the hypotheses of the main theorem are fulfilled, hence its conclusion holds.
\end{corollary}

For the proof, see Example~\ref{ex1}.
Such a bound allows  coercive drift coefficients typical in the theory of diffusion
processes and is a simple algebraic condition, a verification of which does not use any information about
the unknown  solution to the Cauchy problem. Note that if no information about the solution
$\{\mu_t\}$ is given, then our result applies to $A$ and $b$ of linear growth, more
precisely, the function $\|A(t,x)\|+|\langle b(t,x),x\rangle |$
can be of quadratic growth.
Note also that the superposition principle holds for nonlinear equations
(see Remark~\ref{re3.4}).

\section{Auxiliary results}

For the proof of the main result we need some auxiliary assertions.
The next lemma is a simple consequence of the fact that $\mu_t$ is the law
of $\omega(t)$ under~$P_\nu$, but it will be applied repeatedly below.

\begin{lemma}\label{lem1}
Let $g\ge 0$ be a bounded Borel function on $\Omega_d$ and let $f\ge 0$ be a Borel function
on $\mathbb{R}^d$ integrable with respect to the measure $\mu_t$ for some $t\in (0,T]$.
If $P_\nu$ is a probability measure on $\Omega_d$ with property~{\rm(iii)} above,
then
\begin{equation}\label{ek1}
\int_{\Omega_d} f(\omega(t)) g(\omega)\, P_\nu(d\omega)
\le \sup_\omega g(\omega)\int_{\mathbb{R}^d} f(x)\, \mu_t(dx).
\end{equation}
\end{lemma}

The following proposition not only provides an important a priori estimate in the spirit of classical Lyapunov functions
(see \cite{book}, where a variety  of similar results can be found), but also contains an interesting new result: the integrability
of $|LV|$ with respect to the solution.

\begin{proposition}\label{pr1}
Suppose that $\{\mu_t\}$ is a solution of the Cauchy problem $\partial_t\mu_t=L^{*}\mu_t$ with
$\mu_0=\nu$ and there exists a nonnegative function $V\in C^2(\mathbb{R}^d)$ along with a
measurable nonnegative function $W$ such that $V\in L^1(\nu)$
and for some numbers $C\ge 0$ and $\tau\in(0, T]$ one has
$$
\lim\limits_{|x|\to+\infty}V(x)=+\infty, \quad LV(t, x)\le W(t, x)+CV(x), \quad
\int_0^{\tau}\int_{\mathbb{R}^d} W(t, x)\, \mu_t(dx)\,dt<\infty.
$$
Then
$$
\int_{\mathbb{R}^d} V\, d\mu_t\le \biggl(\int_{\mathbb{R}^d} V\,d\nu+\int_0^{\tau}\int_{\mathbb{R}^d} W\,d\mu_s\,ds
\biggr)e^{Ct} \quad \forall t\in[0, \tau],
$$
$$
\int_0^{\tau}\int_{\mathbb{R}^d} |LV|\,d\mu_t\,dt\le 2e^{C\tau}
\biggl(\int_0^\tau\int_{\mathbb{R}^d} W\,d\mu_s\,ds+\int_{\mathbb{R}^d} V\,d\nu\biggr).
$$
\end{proposition}
\begin{proof}
Let $\zeta_N\in C_b^\infty(\mathbb{R})$ be such
that $\zeta_N(t)=t$ if $t\le N-1$, $\zeta_N(t)=N$ if $t>N+1$ and $0\le \zeta_N'\le 1$, $\zeta_N''\le 0$.
In the proof we omit indication of $\mathbb{R}^d$ in integration over the whole space.
Since
$$L\zeta_N(V)=\zeta'_N(V)L(V)+\zeta''_N(V)|\sqrt{A}\nabla V|^2\le \zeta_N'(V)LV,
$$
there holds the inequality
$$
\int\zeta_N(V)\,d\mu_t\le \int\zeta_N(V)\,d\nu+\int_0^t\int
\zeta_N'(V)LV\,d\mu_s\,ds.
$$
Therefore,
$$
\int\zeta_N(V)\,d\mu_t\le\int\zeta_N(V)\,d\nu+\int_0^t\int \zeta_N'(V)
(W+CV)\,d\mu_s\,ds
$$
and
$$
\int\zeta_N(V)\,d\mu_t\le\int V\,d\nu+\int_0^\tau\int W\,d\mu_s\,ds+
C\int_0^t\int \zeta_N(V)\,d\mu_s\,ds.
$$
The announced bound on the integral of $V$ against $\mu_t$ is obtained with the aid
of Gronwall's  inequality by passing to the limit as $N\to\infty$.

Next we write the first inequality in the following form:
$$
\int\zeta_N(V)\,d\mu_t+\int_0^t\int \zeta_N'(V)(LV)^{-}\,d\mu_s\,ds
\le \int\zeta_N(V)\,d\nu
+\int_0^t\int \zeta_N'(V)(LV)^{+}\,d\mu_s\,ds,
$$
where $(LV)^{+}=\max\{LV, 0\}$, $(LV)^{-}=\max\{-LV, 0\}$ and $LV=(LV)^{+}
-(LV)^{-}$.

Since $(LV)^{+}\le W+CV$, we have
$$
\int\zeta_N(V)\,d\mu_t+\int_0^t\int \zeta_N'(V)(LV)^{-}\,d\mu_s\,ds\le
\int\zeta_N(V)\,d\nu
+\int_0^t\int W+CV\,d\mu_s\,ds.
$$
On account of the obtained estimate on $V$ we arrive at the inequality
$$
\int_0^{\tau}\int (LV)^{-}\,d\mu_s\,ds\le e^{C\tau}\biggl(\int_0^{\tau}
\int W\,d\mu_s\,ds+\int V\,d\nu\biggr),
$$
 which yields the announced estimate on the integral of $|LV|=(LV)^{+}+(LV)^{-}$.
\end{proof}

\begin{example}\label{ex1}
\rm
If $\log(1+|x|^2)\in L^1(\nu)$
and
$$
\|A(t,x)\|\le C+C|x|^2\log(1+|x|^2),
\quad
\langle b(t,x), x\rangle \le C+C|x|^2\log(1+|x|^2),
$$
then for some number $C_1$ we have
$$L(\log(1+|x|^2))\le C_1+C_1\log(1+|x|^2),
\quad
\|A(t,x)\|/(1+|x|^2)\le C_1+C_1\log(1+|x|^2),
$$
$$
|\langle b(t,x), x\rangle |/(1+|x|^2)\le |L(\log(1+|x|^2)|+C_1+C_1\log(1+|x|^2).
$$
Hence by Proposition~\ref{pr1}
the functions $\log(1+|x|^2)$ and $|L(\log(1+|x|^2)|$
are integrable on $[0, T]\times\mathbb{R}^d$ with respect to~$\mu_t\,dt$
and condition~(\ref{con-new}) is fulfilled.
\end{example}

\begin{proposition}\label{pr2}
Suppose that $V\in C^2(\mathbb{R}^d)$ and
\mbox{$\lim\limits_{|x| \to+\infty}V(x)=+\infty$}.

\, {\rm (i)} \,  There exists
a function $\theta\in C^2(\mathbb{R})$ such that $\theta(V)\in L^1(\nu)$ and
$$
\theta\ge 0, \quad \theta(0)=0, \quad
0\le\theta'(t)\le 1,
\quad \theta''\le 0,
\quad
\lim\limits_{t\to+\infty}\theta(t)=+\infty.
$$

\, {\rm (ii)} \, Assume that for some $\tau\in(0, T]$ one has
$$
\int_0^{\tau}\int_{\mathbb{R}^d} \Bigl( |\sqrt{A}\nabla V|^2+|LV|\Bigr)\,d\mu_t\,dt<\infty
$$
and that $\theta$ satisfies all assumptions listed in {\rm (i)} and $\theta(V)\in L^1(\nu)$.
Then $\theta(V)$ satisfies the following inequality:
\begin{multline*}
\int_0^{\tau}\int_{\mathbb{R}^d} \Bigl( |\sqrt{A}\nabla\theta(V)|^2+|L\theta(V)|\Bigr)\,d\mu_t\,dt
\\
\le 2e^{C\tau} \biggl(\int\theta(V)\,d\nu+\int_0^{\tau}
\int_{\mathbb{R}^d} \Bigl( |\sqrt{A}\nabla V|^2+|LV|\Bigr)\,d\mu_t\,dt\biggr).
\end{multline*}
\end{proposition}
\begin{proof}
In the  proof of Proposition 7.1.8 in \cite{book} it was shown that there is a
 function $\theta\in C^2(\mathbb{R})$ such that $\theta\ge 0$, $\theta(0)=0$,
$0\le\theta'(t)\le 1$, $\theta''\le 0$
and $\theta(V)\in L^1(\nu)$. Then
$$
|\sqrt{A}\nabla \theta(V)|^2=|\theta'|^2|\sqrt{A}\nabla V|^2\le|\sqrt{A}
\nabla V|^2,
$$
$$
L \theta(V)=\theta''(V)|\sqrt{A}\nabla V|^2+\theta'(V)LV\le |LV|.
$$
Applying Proposition \ref{pr1} with $W=|LV|$ and $C=0$ we obtain our claim.
\end{proof}

The next assertion enables us to estimate the measure of a ball in the path space
with the aid of the function~$V$.

\begin{proposition}\label{pt-ball}
Let $\tau\in(0, T]$.
Suppose that $\{\mu_t\}$ is a solution to the  Cauchy problem $\partial_t\mu_t=L^{*}\mu_t$
on $[0,\tau]$ with
$\mu_0=\nu$ and that there exists a Borel probability measure $P_{\nu}$ on
$C([0,\tau],\mathbb{R}^d)$
such that {\rm (i)}, {\rm (ii)} and {\rm (iii)} are fulfilled. Suppose also that there is a
nonnegative function $V\in C^2(\mathbb{R}^d)$ with $\lim\limits_{|x|\to\infty}V(x)=+\infty$
such that $V\in L^1(\nu)$ and
$$
\int_0^\tau \int_{\mathbb{R}^d}
\Bigl(|\sqrt{A}\nabla V|^2+|LV|\Bigr)\,d\mu_t\,dt<\infty.
$$
Then for every $q>0$ one has
$$
P_{\nu}\Bigl(\omega\colon \sup_{t\in[0, T]}V(\omega(t))\ge q\Bigr)
\le
\frac{2}{q}\biggl(\int_{\mathbb{R}^d}V\,d\nu+\int_0^\tau \int_{\mathbb{R}^d}
\Bigl( |\sqrt{A}\nabla V|^2+|LV|\Bigr)\,d\mu_s\,ds\biggr).
$$
\end{proposition}
\begin{proof}
Using the function $\zeta_N(V)$ in place of $V$ and the assumption about the integrability
of the function $LV$ one can verify that
$$V(\omega(t))-V(\omega(0))
-\int_0^t
LV(s,\omega(s))\,ds
$$
is a martingale with the quadratic variation
$$
\int_0^t |\sqrt{A}\nabla V(s,\omega(s))|^2\,ds.
$$
By Doob's  inequality we have
\begin{multline*}
P_{\nu}\biggl(\omega\colon \sup_{t\in[0, \tau]}\biggl|V(\omega(t))
-V(\omega(0))-\int_0^t LV(s,\omega(s))\,ds\biggr|\ge q\biggr)
\\
\le
\frac{1}{q}
\int_0^\tau\int_{\mathbb{R}^d} |\sqrt{A}\nabla V|^2\,d\mu_s\,ds.
\end{multline*}
Since
\begin{multline*}
P_{\nu}\Bigl(\omega\colon \sup_{t\in[0, \tau]}V(\omega(t))\ge q\Bigr)
\\
\le P_{\nu}\biggl(\omega\colon \sup_{t\in[0,\tau]}
\biggl|V(\omega(t))-V(\omega(0))-\int_0^tLV(s,\omega(s))
\,ds\biggr|\ge q/2\biggr)
\\
+
P_{\nu}\biggl(\omega\colon \sup_{t\in[0,\tau]}
\biggl|V(\omega(0))+\int_0^tLV(s,\omega(s))\,ds\biggr|
\ge q/2\biggr),
\end{multline*}
we obtain
$$
P_{\nu}\Bigl(\omega \colon \sup_{t\in[0,\tau]}V(\omega(t))\ge q\Bigr)\le
\frac{2}{q}\biggl(\int_{\mathbb{R}^d} V\,d\nu
+\int_0^\tau\int_{\mathbb{R}^d} \Bigl(|\sqrt{A}\nabla V|^2+|LV|\Bigr)
\,d\mu_s\,ds\biggr),
$$
which completes the proof.
\end{proof}

\section{Proof of the main result}

The proof of the theorem follows the scheme used by Figalli~\cite{F} and Trevisan~\cite{TREV}.
However, there are some differences:
Trevisan's result is not applicable even for smooth coefficients
without global integrability of the coefficients. Here we substantionally used some recent
results on the uniqueness of probbaility solutions to Fokker--Planck--Kolmogorov equations
from \cite{book} and \cite{ManSh}.
When reducing the general case to that of smooth coefficients (as also Figalli and Trevisan did),
we encounter two problems: 1)~it is necessary to control that the solutions
with smoothed coefficients converge to the considered solution, which is not automatic due
to the lack of uniqueness, 2)~for a~priori estimates it is necessary to keep
condition~(\ref{con-new}) uniformly. The first problem is overcome by using
the smoothing involving not only the space variable, but also the time.
The second problem is solved with the aid of the equation itself,
namely, we estimate the integral of $\langle \beta(t,x), x\rangle$
for the approximating drift $\beta$ by means of some integral of the diffusion matrix.
Note also that before picking a common compact set of measure close to~$1$
for the corresponding measures on $\Omega_d$
we first pick a common ball of measure close to~$1$.
Finally,
we verify  that for passing to the limit our local integrability
conditions on the coefficients are sufficient.

\begin{proof}[Proof of Theorem \ref{tmain}]
First we assume that all our hypotheses hold on some larger interval $[0,T_1]$, $T_1>T$,
and at the last step explain how to obtain a representation on $[0,T]$ without that assumption.
Set
$$
V(x)=\log(1+|x|^2).
$$
 Take a function $\theta$ such that $\theta(\log(3+2|x|^2))\in L^1(\nu)$
and all conditions from (i) in~Proposition~\ref{pr2} are fulfilled.
According to Proposition~\ref{pr1} there is a number $N_\theta$ such that
$$
\sup_{t\in[0, T_1]}\int \theta(\log(3+2|x|^2))\, \mu_t(dx)\le N_\theta.
$$
Note that the coefficients $3$ and $2$ are only technical things.

{\bf I. Justification of the replacement of $\mu_t$ by $\mu_t^{\delta}=\mu_{t+\delta}$.}

Passing from $\mu_t$ to $\mu_{t+\delta}$ enables us to smmoth solutions not only
with respect to~$x$, but also with respect to~$t$.

Suppose that for every $\delta\in(0, T_1-T)$ there exists a measure $P_{\delta}$
on $\Omega_d$ satisfying
(i), (ii) and (iii) with the coefficients
$$
A_{\delta}(t, x)=A(t+\delta, x), \quad b_{\delta}(t, x)=b(t+\delta, x)
$$
and the corresponding operator
$$
L_\delta f=a_{\delta}^{ij}\partial_{x_i}\partial_{x_j}u+b^i_\delta \partial_{x_i} u
$$
such that  $\mu^{\delta}_t=\mu_{t+\delta}$ solves the  Cauchy problem with initial condition
$\mu^{\delta}_0=\mu_{\delta}$.
We show that it is possible to extract from $P_{\delta}$ a weakly convergent sequence
$P_{\delta_n}$ with $\delta_n\to 0$
such that its limit is a solution to the original martingale problem and gives a representation for
the original solution~$\mu_t$.
We observe that $\mu_t^{\delta}$ converges weakly to $\mu_t$ as $\delta\to 0$ for every $t$ by the
continuity of the mapping $t\mapsto\mu_t$. Omitting again indication of $\mathbb{R}^d$
when integrating over the whole space, we have
$$
\int \theta(V)\,d\mu^{\delta}_t\le \sup_{t\in[0, T_1]}\int\theta(V)\,d\mu_t\le N_\theta.
$$
Moreover,
$$
\int_0^T\int \Bigl(|\sqrt{A_{\delta}}\nabla V|^2+|L_{\delta}V|\Bigr)\,d\mu_t^{\delta}\,dt\le
\int_0^{T_1}\int \Bigl(|\sqrt{A}\nabla V|^2+|LV|\Bigr)\,d\mu_t\,dt\le C_1,
$$
where $C_1$ does not depend on $\delta$.
By Proposition \ref{pr2}
$$
\int_0^T\int \Bigl(|\sqrt{A_{\delta}}\nabla\theta(V)|^2+|L_{\delta}\theta(V)|\Bigr)d\mu_t^{\delta}\,dt\le
C(T)(N_\theta+C_1).
$$
By Proposition \ref{pt-ball} applied to the function $\theta(V)$, for every $\varepsilon>0$, there exists $R>0$
such that for all $\delta\in(0, T_1-T)$ we have
$$
P_{\delta}\bigl(\omega\colon \|\omega\|\le R\bigr)\ge 1-\varepsilon.
$$
We now need two results from \cite{TREV}. The first one is Theorem~A2.
Let $\Theta\colon [0,+\infty)\to [0,+\infty)$ be a lower semicontinuous function
and let $\Theta_1, \Theta_2 \colon [0,+\infty)\to [0,+\infty)$ be convex functions
such that $\Theta_2(2x)\le C\Theta_2(x)$, $\Theta_1(0)=\Theta_2(0)=0$ and
$$
\lim\limits_{x\to+\infty}
\Theta(x)=\lim\limits_{x\to+\infty}\frac{\Theta_1(x)}{x}
=\lim\limits_{x\to+\infty}\frac{\Theta_2(x)}{x}=+\infty.
$$
Then there is a compact function $\Psi\colon C[0, T]\to [0,+\infty]$, i.e.,
the sets $\{\Psi\le R\}$ are compact for finite $R$, such that, whenever
$\{\alpha_t\}$, $\{\beta_t\}$, $\varphi=\{\varphi_t\}$ are progressively measurable
processes on a filtered probability space $(\Omega,(\mathcal{F}_t)_{t\in [0, T]},P)$ for which
$$
M_t:=\varphi_t-\int_0^t \beta_s\, ds \quad \hbox{and} \quad
M_t^2-\int_0^t \alpha_s\, ds
$$
are $P$-a.s. continuous local martingales and $\alpha_t\ge 0$ a.s., one has
$$
\mathbb{E}\Psi(\varphi)\le
\mathbb{E}\biggl[ \Theta(\varphi_0)+\int_0^T [\Theta_1(|\beta_t|)+\Theta_2(\alpha_t)]\, dt\biggr].
$$
Next, according to \cite[Corollary~A5]{TREV}, if $\eta\in\mathcal{P}(C[0, T],\mathbb{R}^d)$
is a solution to the martingale problem associated with an elleiptic operator~$L$
(not necessarily our operator),
then, for every function $f\in  C_0^\infty(\mathbb{R}^d)$ and the marginal distributions
$\eta_t$ for $\eta$ it holds
$$
\int \Psi(f)\, d\eta
\le \int \Theta(|f(0,x)|\, \eta_0(dx)+
\int_0^T \int  \Bigl[\Theta_1(|Lf|) + \Theta_2(|\sqrt{A}\nabla f|^2|)\Bigr] \, d\eta_t\, dt,
$$
where $\Psi$, $\Theta$, $\Theta_1$ and $\Theta_2$ are the same as above
and $\Psi(f)(\omega)=\Psi(f(\omega(\cdot)))$. Note that the right-hand side is finite
due to our hypotheses on $A$ and $b$ and the compactness of support of~$f$ in~$x$.

Now take $\psi_R\in C^{\infty}_0(\mathbb{R}^d)$ with $\psi_R(x)=1$ if $|x|\le R$,
$\psi_R(x)=0$ if $|x|>2R$. Let us apply the cited corollary to the function
$f_i(x)=x_i\psi_R(x)$ independent of~$t$. Denoting by $\mathbb{E}_{P_{\delta}}$
the integral with respect to~$P_\delta$, we obtain the estimate
\begin{multline*}
\mathbb{E}_{P_{\delta}}\Psi(f_i)\le \int\Theta(x_i\psi_R(x))
\,\mu_{\delta}(dx)
\\
+
\int_0^T\int \Bigl(\Theta_1(|L_{\delta}(x_i\psi_R(x))|)+\Theta_2
(|\sqrt{A_{\delta}}\nabla(x_i\psi_R)|^2)\Bigr)\, \mu_t^{\delta}(dx)\,dt,
\end{multline*}
where for some number $C_2$ independent of $\delta$
the right-hand side is estimated by
$$
\sup_x|\Theta(x_i\psi_R(x))|+C_2\int_0^{T_1}\int_{|x|\le 2R}\Bigl(\Theta_1
(\|A(t, x)\|+|b(t, x)|)+\Theta_2(\|A(t, x)\|)\Bigr)\, \mu_t(dx)\,dt,
$$
which does not depend on $\delta$.
As in \cite{TREV}, we consider the compact function
$$
\Psi_d(\omega)=\sum_{i=1}^d\Psi(\omega_i)
$$
on $\Omega_d$. We have
$$
\mathbb{E}_{P_{\delta}} (I_{\|\omega\|\le R}\Psi_d) =\mathbb{E}_{P_{\delta}}
[I_{\|\omega\|\le R}\Psi(f_i)]\le
\sum_{i=1}^d\mathbb{E}_{P_{\delta}}\Psi(f_i)\le C_3(R),
$$
where $C_3(R)$ depends on $R$, but does not depend on $\delta$.
Taking a sufficiently large number $M$ we conclude that for the compact set
$$
K=\{\omega\colon \Psi_d(\omega) \le M, \ \|\omega\|\le R\}
$$
in $\Omega_d$
there holds the estimate
$$
P_{\delta}(K)\ge 1-2\varepsilon.
$$
Therefore, the family of measures $P_{\delta}$ contains a sequence $P_{\delta_k}$ with
$\delta_{k}\to 0$ weakly converging to some probability measure $P$.
Let us verify that the measure $P$ satisfies (i), (ii) and (iii).
The first and last properties are obtained in the limit as $\delta_k\to 0$
in the equality
$$
\int_{\Omega_d} f(\omega(t))\, P_{\delta_k}(d\omega)=\int f(x)\,
\mu_{t+\delta_k}(dx) \quad \forall f\in C_0^{\infty}(\mathbb{R}^d).
$$
For the proof of the second (martingale) property we have to show that for every bounded continuous
function
$$g\colon \Omega_d\to\mathbb{R}$$
that is measurable with respect to the $\sigma$-algebra $\mathcal{F}_s$ there holds the equality
$$
\int_{\Omega_d}\Bigl[f(\omega(t))-f(\omega(s))-\int_s^t
Lf(\tau,\omega(\tau))\, d\tau\Bigr]g(\omega)\, P(d\omega)=0, \quad t\ge s,
$$
where $f\in C_0^{\infty}(\mathbb{R}^d)$.
To  this end, it suffices to show that
$$
\lim\limits_{\delta_k\to 0}\int_{\Omega_d}\biggl[\int_s^t
L_{\delta}f(\tau, \omega(\tau))\, d\tau\biggr]g(\omega)\,P_{\delta}(d\omega)
=
\int_{\Omega_d}\biggl[\int_s^t Lf(\tau,\omega(\tau))\, d\tau\biggr]g(\omega)\,P(d\omega),
$$
because convergence of the integrals of $[f(\omega(t))-f(\omega(s))]g(\omega)$ is obvious by the
continuity of this function on~$\Omega_d$.
Let $q^{ij}, z^i\in C^{\infty}([-1, T_1]\times\mathbb{R}^d)$ and
$$
\widetilde{L}=q^{ij}\partial_{x_i}\partial_{x_j}+z^i\partial_{x_i}
$$
and let $\widetilde{L}_\delta$ be the corresponding operator with the time-shifted coefficients
$q^{ij}(t+\delta,x)$ and  $z^i(t+\delta,x)$.

It is clear from (\ref{ek1}) that the difference
$$
\int_{\Omega_d}\biggl[\int_s^t \widetilde{L}f(\tau,\omega(\tau))\, d\tau
\biggr]g(\omega)\,P_{\delta}(d\omega)
-
\int_{\Omega_d}\biggl[\int_s^t \widetilde{L}f(\tau,\omega(\tau))\, d\tau
\biggr]g(\omega)\,P(d\omega)
$$
tends to zero. Moreover, by Lemma~\ref{lem1} the expression
$$
\biggl|\int_{\Omega_d}\biggl[\int_s^t (\widetilde{L}-L)
f(\tau,\omega(\tau))\biggr]g(\omega)\,P(d\omega)\biggr|
$$
is estimated by
$$
\int_0^T\int \Bigl(|a^{ij}-q^{ij}||\partial_{x_i}\partial_{x_j}f|
+|b^i-z^i||\partial_{x_i}f|\Bigr)\,d\mu_t\,dt,
$$
which can be made arbitrarily small by a suitable choice of $q^{ij}$ and $z^i$
approximating $a^{ij}$ and $b^i$ in $L^1$ with respect to the measure
$\mu_t\, dt$ on $[0, T_1]\times U$, where $U$ is a ball containing the support of~$f$.
Since the functions $(\widetilde{L}-\widetilde{L}_{\delta})f(t,x)$ converge to zero uniformly on
$[0, T]\times\mathbb{R}^d$ as $\delta\to 0$, the expression
$$
\int_{\Omega_d}\biggl[\int_s^t (\widetilde{L}
-\widetilde{L}_{\delta})f(\tau,\omega(\tau))\, d\tau\biggr]g(\omega)\,P_{\delta}(d\omega)
$$
tends to zero as $\delta\to 0$. Finally, we observe that the expression
$$
\int_{\Omega_d}\biggl[\int_s^t (\widetilde{L}_{\delta}
-L_{\delta})f(\tau,\omega(\tau))\, d\tau\biggr]g(\omega)\,P_{\delta}(d\omega)
$$
is estimated by
$$
\int_0^{T_1}\int \Bigl(|a^{ij}-q^{ij}||\partial_{x_i}
\partial_{x_j}f|+|b^i-z^i||\partial_{x_i}f|\Bigr)\,d\mu_t\,dt,
$$
which can be made arbitrarily small by a suitable choice of $q^{ij}$ and $z^i$ as above.
Thus, we have verified (i), (ii), (iii) for $P$.
Therefore, for completing the proof of the theorem it suffices to show
that for each fixed $\delta>0$ there exists a suitable  measure $P_\delta$
for the solution $\mu_t^{\delta}$ to the  Cauchy problem $\partial_t\mu_t^{\delta}
=L_{\delta}^{*}\mu_t^{\delta}$ with $\mu_0^{\delta}=\mu_{\delta}$.

\vskip .1in
{\bf II. Smoothing of the coefficients and verification of the conditions with the Lyapunov function}.

Let us fix $\delta \in (0, T_1-T)$.
Let $\zeta\in C^{\infty}([0, +\infty))$, $0\le\zeta\le 1$,
$\zeta'\le 0$, $\zeta(t)=1$ if $t<1$ and $\zeta(t)=0$ if $t>2$.
Set
$$
\eta(t)=\int_t^{+\infty}\zeta(s)\,ds.
$$
It is clear that $\eta\ge 0$, $\eta(t)=0$ if $t>2$ and $\eta'(t)=-\zeta(t)$.
Let $c_1$ and $c_2$ be numbers such that
$$
c_1\int_{\mathbb{R}^d} \zeta(|x|^2)\,dx=1, \quad c_2\int_{\mathbb{R}} \zeta(|t|^2)\,dt=1.
$$
For every $\varepsilon$ with  $0<\varepsilon<\min\{\delta/16, 1/2\}$ set
$$
h_{\varepsilon}(t, x)=c_1c_2\varepsilon^{-d-1}\zeta(|t|^2/
\varepsilon^2)\zeta(|x|^2/\varepsilon^2).
$$
Let $\gamma$ be the  standard Gaussian density on $\mathbb{R}^d$. Set
$$
\sigma^{\varepsilon}(t, x)=\varepsilon\gamma(x)+(1-\varepsilon)
\int\int h_{\varepsilon}(t-s, x-y)\mu_s^{\delta}(dy)\,ds,
$$
where the integration is formally taken over all of $\mathbb{R}^{d+1}$.
However, we take into account that the
function $h_{\varepsilon}(t-s, x-y)$ vanishes if $s\le -\delta/2$
or $s\ge T+\delta/2$, so that actually the integration in $s$ is taken within the limits
 $-\delta/2$ and $T+\delta/2$ and for such $s$ the measures $\mu_s^{\delta}$
are defined. It is clear that $\sigma^{\varepsilon}>0$ and
$$
\int\sigma^{\varepsilon}(t, x)\,dx=1.
$$
In addition, for every function $f\in C_0^{\infty}(\mathbb{R}^d)$ and
every $t\in[0, T]$ we have
$$
\lim_{\varepsilon\to 0}\int f(x)\sigma^{\varepsilon}(t, x)\,dx
=\int f(x)\,\mu_t^{\delta}(dx).
$$
Indeed,
\begin{multline*}
\int f(x)\sigma^{\varepsilon}(t, x)\,dx=
\varepsilon\int f\gamma\, dx+
\\
+(1-\varepsilon)\int c_2\varepsilon^{-1}\zeta(|t-s|^2/\varepsilon^2)
\int \biggl(\int f(x)c_1\varepsilon^{-d}\zeta(|x-y|^2/\varepsilon^2)
\,dx\biggr)\,\mu_s^{\delta}(dy)\,ds.
\end{multline*}
Since
$$
\sup_y\biggl|\int f(x)c_1\varepsilon^{-d}\zeta(|x-y|^2/\varepsilon^2)
\,dx-f(y)\biggr|\le
\varepsilon \sup|\nabla f|c_1\int|x|\zeta(|x|^2)\,dx ,
$$
 it suffices to show that the limit of the expression
$$
\int c_2\varepsilon^{-1}\zeta(|t-s|^2/\varepsilon^2)
\int f(y)\, \mu_s^{\delta}(dy)\,ds
$$
is equal to the  integral of $f$ against $\mu_t^{\delta}$.
This follows immediately by the continuity of the  function
$$s\mapsto \int f(y)\,\mu_s^{\delta}(dy).$$
Thus, for every $t\in[0, T]$ the measures
 $\sigma^{\varepsilon}(t, x)\,dx$  converge weakly to $\mu_t^{\delta}$.

We recall  that $b_{\delta}(t, x)=b(t+\delta, x)$
and $A_{\delta}(t, x)=A(t+\delta, x)$.
Set
$$
\beta_{\varepsilon}^i(t, x)
=\frac{1-\varepsilon}{\sigma^{\varepsilon}(t, x)}
\int\int b_{\delta}^i(s, y)h_{\varepsilon}(t-s, x-y)\, \mu_s^{\delta}(dy)\,ds,
$$
$$
\alpha^{ij}_{\varepsilon}(t, x)
=\frac{1-\varepsilon}{\sigma^{\varepsilon}(t, x)}
\int\int a_{\delta}^{ij}(s, y)h_{\varepsilon}(t-s, x-y)
\, \mu_s^{\delta}(dy)\,ds.
$$
Recall that $\gamma$  is the standard Gaussian density on $\mathbb{R}^d$.
We shall deal with the  operator
$$
\mathcal{L}_{\varepsilon}u(t, x)={\rm trace}(\alpha_{\varepsilon}(t, x)D^2u(x))
+\langle\beta_{\varepsilon}(t, x), \nabla u(x)\rangle
+\frac{\varepsilon\gamma(x)}{\sigma^{\varepsilon}(t, x)}\bigl(\Delta u(x)-\langle x, \nabla u(x)\rangle\bigr),
$$
which should not be confused with the previously defined $L_\varepsilon$;
moreover, $\mathcal{L}_{\varepsilon}$  depends also on~$\delta$, which is now fixed
and is not shown in this notation. Set also
$$
\mathcal{A}_{\varepsilon}=\alpha_{\varepsilon}+\frac{\varepsilon\gamma(x)}{\sigma^{\varepsilon}(t, x)}I.
$$
It is readily seen  that $\sigma^{\varepsilon}$ solves on
$[0, T]\times\mathbb{R}^d$ the  Cauchy problem
$$
\partial_t\sigma^{\varepsilon}=\mathcal{L}_{\varepsilon}^{*}
\sigma^{\varepsilon}, \quad
\sigma^{\varepsilon}(0, x)=\varepsilon\gamma(x)
+(1-\varepsilon)\int\int h_{\varepsilon}(s, x-y)\, \mu_s^{\delta}(dy)\,ds.
$$
We now investigate the expression $\langle \beta_{\varepsilon}(t, x), x\rangle$.
We have
\begin{multline*}
\langle \beta_{\varepsilon}(t, x), x\rangle=
\frac{1-\varepsilon}{\sigma^{\varepsilon}(t, x)}
\int\int\langle b_{\delta}(s, y), y\rangle h_{\varepsilon}
(t-s, x-y)\mu_s^{\delta}(dy)\,ds
\\
+\frac{1-\varepsilon}{\sigma^{\varepsilon}(t, x)}
\int\int\langle b_{\delta}(s, y), x-y\rangle
h_{\varepsilon}(t-s, x-y)\, \mu_s^{\delta}(dy)\,ds.
\end{multline*}
Let us consider the expression
\begin{multline*}
\int\int\langle b_{\delta}(s, y), x-y\rangle
h_{\varepsilon}(t-s, x-y)\, \mu_s^{\delta}(dy)\,ds
\\
=\int c_2\varepsilon^{-1}\zeta(|t-s|^2/\varepsilon)\biggl(
\int\langle b_{\delta}(s, y), x-y\rangle c_1\varepsilon^{-d}
\zeta(|x-y|^2/\varepsilon^2)\, \mu_s^{\delta}(dy)\biggr)\,ds.
\end{multline*}
We observe that $\zeta=-\eta'$ and
$$
\langle b_{\delta}(s, y), x-y\rangle c_1\varepsilon^{-d}
\zeta(|x-y|^2/\varepsilon^2)=
-2^{-1}c_1\varepsilon^2\Bigl\langle b_{\delta}(s, y),
\nabla_x\bigl(\varepsilon^{-d}\eta(|x-y|^2/\varepsilon^2)\bigr)\Bigr\rangle.
$$
Therefore,
\begin{multline*}
\int\int\langle b_{\delta}(s, y), x-y\rangle
h_{\varepsilon}(t-s, x-y)\, \mu_s^{\delta}(dy)\,ds
\\
=-2^{-1}c_1c_2\varepsilon^2\partial_{x_i}
\biggl(\int\int b_{\delta}^i(s, y)
\varepsilon^{-d-1}\zeta(|t-s|^2/\varepsilon^2)\eta(|x-y|^2/
\varepsilon^2)\mu_s^{\delta}(dy)\,ds\biggr).
\end{multline*}
Recall that
$$
\partial_t\mu_t^{\delta}=\partial_{x_i}\partial_{x_j}(a_{\delta}^{ij}\mu_t^\delta)-\partial_{x_i}(b^i_{\delta}\mu_t^\delta)
$$
on $(-\delta, T+\delta)\times\mathbb{R}^d$ and for every fix $(t, x)\in[0, T]\times\mathbb{R}^d$
the function
$$
\zeta(|t-s|^2/\varepsilon^2)\eta(|x-y|^2/\varepsilon^2)
$$
has  compact support in $(-\delta, T+\delta)\times\mathbb{R}^d$.
Thus, there holds the equality
\begin{multline*}
-\partial_{x_i}
\biggl(\int\int b_{\delta}^i(s, y)\varepsilon^{-d-1}
\zeta(|t-s|^2/\varepsilon^2)\eta(|x-y|^2/\varepsilon^2)
\, \mu_s^{\delta}(dy)\,ds\biggr)
\\
=
\partial_t\biggl(\int\int\varepsilon^{-d-1}\zeta(|t-s|^2/
\varepsilon^2)\eta(|x-y|^2/\varepsilon^2)\, \mu_s^{\delta}(dy)\,ds\biggr)
\\
-\partial_{x_i}\partial_{x_j}\biggl(\int\int a_{\delta}^{ij}(s, y)
\varepsilon^{-d-1}\zeta(|t-s|^2/\varepsilon^2)\eta(|x-y|^2/\varepsilon^2)
\, \mu_s^{\delta}(dy)\,ds\biggr).
\end{multline*}
We can estimate the terms in the right-hand as follows:
\begin{multline*}
\partial_t\biggl(\int\int\varepsilon^{-d-1}\zeta(|t-s|^2/
\varepsilon^2)\eta(|x-y|^2/\varepsilon^2)\, \mu_s^{\delta}(dy)\,ds\biggr)
\\
\le 2\varepsilon^{-d-3}\int\int|t-s||\zeta'(|t-s|^2/\varepsilon^2)|
\eta(|x-y|^2/\varepsilon^2)\mu_s^{\delta}(dy)\,ds,
\end{multline*}
\begin{multline*}
-\partial_{x_i}\partial_{x_j}\biggl(\int\int a_{\delta}^{ij}(s, y)
\varepsilon^{-d-1}\zeta(|t-s|^2/\varepsilon^2)\eta(|x-y|^2/
\varepsilon^2)\, \mu_s^{\delta}(dy)\,ds\biggr)
\\
=4\varepsilon^{-d-5}\int\int\langle A_{\delta}(x-y), (x-y)\rangle
\zeta(|t-s|^2/\varepsilon^2)\zeta'(|x-y|^2/\varepsilon^2)\,
\mu_s^{\delta}(dy)\,ds
\\
+2\varepsilon^{-d-3}\int\int ({\rm trace}\, A_{\delta}) \zeta(|t-s|^2/
\varepsilon^2)\zeta(|x-y|^2/\varepsilon^2)\, \mu_s^{\delta}(dy)\,ds
\\
\le 2\varepsilon^{-d-3}\int\int ({\rm trace}\, A_{\delta}) \zeta(|t-s|^2/
\varepsilon^2)\zeta(|x-y|^2/\varepsilon^2)\, \mu_s^{\delta}(dy)\,ds.
\end{multline*}
For obtaining the last  inequality we have used that $\zeta'\le 0$ and
$$
\langle A_{\delta}(x-y), (x-y)\rangle\ge 0.
$$
We observe that whenever $|x-y|\le 2\varepsilon\le 1$ one has
$$
\frac{1}{1+|x|^2}\le \frac{3}{1+|y|^2}.
$$
Thus, we have obtained the estimate
\begin{multline*}
\frac{\langle \beta_{\varepsilon}(t, x), x\rangle}{1+|x|^2}\le
\frac{3}{\sigma^{\varepsilon}(t, x)}\int\int\frac{|\langle b_{\delta},
y\rangle|}{1+|y|^2} h_{\varepsilon}(t-s, x-y)\, \mu_s^{\delta}(dy)\,ds
\\
+\frac{c_1c_2\varepsilon^{-d-1}}{\sigma^{\varepsilon}(t, x)}\int\int|t-s|
|\zeta'(|t-s|^2/\varepsilon^2)|\eta(|x-y|^2/\varepsilon^2)
\, \mu_s^{\delta}(dy)\,ds
\\
+\frac{3c_1c_2\varepsilon^{-d-1}}{\sigma^{\varepsilon}(t, x)}
\int\int\frac{{\rm trace}\, A_{\delta}}{1+|y|^2}\zeta(|t-s|^2/
\varepsilon^2)\zeta(|x-y|^2/\varepsilon^2))\, \mu_s^{\delta}(dy)\,ds.
\end{multline*}
Let us denote the right-hand side of this inequality by $W_1(t, x)$ and observe
that $W_1\ge 0$ and
\begin{multline*}
\int_0^{T}\int W_1(t, x)\sigma^{\varepsilon}(t, x)\,dx\,dt\le
C_4\int_{-\delta}^{T+\delta}\int \frac{|\langle b_{\delta}, y\rangle|+
|{\rm trace}\, A_{\delta}|}{1+|y|^2}\,\mu_s^{\delta}(dy)\,ds
\\
+
c_1c_2\int\int |t|\, |\zeta'(t^2)|\eta(|x|^2)\,dx\,dt,
\end{multline*}
where $C_4$ does not depend on $\varepsilon$.
The function
$$
\frac{|\alpha^{ij}_{\varepsilon}(t, x)|}{1+|x|^2}
$$
is estimated by
$$
W_2(t, x):=\frac{3}{\sigma^{\varepsilon}(t, x)}\int\int
\frac{|a^{ij}_{\delta}(s,y)|}{1+|y|^2}h_{\varepsilon}(t-s, x-y)\, \mu_s^{\delta}(dy)\,ds.
$$
We observe that $W_2\ge 0$ and
$$
\int_0^T\int W_2(t, x)\sigma^{\varepsilon}(t, x)\,dx\,dt\le
3\int_{-\delta}^{T+\delta}\int \frac{|a^{ij}_{\delta}(s,y)|}{1+|y|^2}
\,\mu_s^{\delta}(dy)\,ds.
$$
Set
$$
W_3(t, x)=\frac{\varepsilon\gamma(x)}{(1+|x|^2)\sigma^{\varepsilon}(t, x)}.
$$
Note that
$$
\int_0^T\int W_3(t, x)\sigma^{\varepsilon}(t, x)\,d\,dt\le T\int(1+|x|^2)^{-1}\gamma(x)\,dx.
$$
Thus, we arrive at the estimates
$$
\mathcal{L}_{\varepsilon}\log(1+|x|^2)\le C_5(W_1+W_2+W_3), \quad \Bigl|\sqrt{\mathcal{A}_{\varepsilon}}
\nabla\log(1+|x|^2)\Bigr|^2\le C_5(W_2+W_3),
$$
where $C_5$ does not depend on $\varepsilon$.
Note that for our function $V(x)=\log(1+|x|^2)$ we have
$$
\mathcal{L}_{\varepsilon}\theta(V)=\theta''(V)|\sqrt{\mathcal{A}_{\varepsilon}}\nabla V|^2
+\theta'(V)\mathcal{L}_{\varepsilon}V
\le C_5(W_1+W_2+W_3),
$$
and
$$
|\sqrt{\mathcal{A}_{\varepsilon}}\nabla \theta(V)|^2\le |\sqrt{\mathcal{A}_{\varepsilon}}\nabla V|^2\le C_5(W_2+W_3).
$$
Moreover,
\begin{multline*}
\int \theta(V(x))\sigma^{\varepsilon}(0, x)\,dx\le \varepsilon\int V(x)\gamma(x)\,dx
\\
+
(1-\varepsilon)\int \int\int h_{\varepsilon}(s, x-y)V(x)\,\mu_s^{\delta}(dy)\,dx\,ds.
\end{multline*}
Note that $\log(1+|x|^2)\le \log(1+2|x-y|^2+2|y|^2)\le \log(3+2|y|^2)$ if $|x-y|\le 1$.
Recall also that
$$
\sup_{t\in[0, T]}\int\theta(\log(3+2|y|^2))\, \mu_t^{\delta}(dy)\le N_\theta.
$$
Hence
$$
\int \theta(V(x))\sigma^{\varepsilon}(0, x)\,dx\le \int V(x)\gamma(x)\,dx+N_\theta.
$$
Applying Proposition \ref{pr1} we obtain
\begin{equation}\label{integr}
\int_0^T\int\Bigl(|\mathcal{L}_{\varepsilon}\theta(V)|
+|\sqrt{\mathcal{A}_{\varepsilon}}\nabla \theta(V)|^2\Bigr)\sigma^{\varepsilon}
(t, x)\,dx\,dt\le C_6,
\end{equation}
where $C_6$ does not depend on $\varepsilon$ (recall that $V(x)=\log(1+|x|^2)$).

\vskip .1in
{\bf III. The representation of $\sigma^{\varepsilon}$ by a solution to the
martingale problem with $\mathcal{L}_{\varepsilon}$.}

According to \cite[Theorem 3.5]{ManSh} (see also \cite[Theorem 9.4.3]{book}),
the function $\sigma^{\varepsilon}$ is a unique subprobability solution to the Cauchy problem.
It is shown in \cite{ManSh} and it is very important in our situation that
the integrability condition (\ref{integr}) just for one probability solution $\sigma^{\varepsilon}$
implies the uniqueness in the class of all sub-probability solutions.

We show that there exists a solution to the martingale problem with the operator
$\mathcal{L}_{\varepsilon}$ and initial condition $\sigma^{\varepsilon}(0, x)\,dx$.
Let $\varphi_N(x)=\varphi(x/N)$, $0\le\varphi\le 1$, $\varphi\in C^{\infty}_0(\mathbb{R}^d)$
and $\varphi(x)=1$ if $|x|<1$. Set
$$\mathcal{L}_{\varepsilon}^N=\varphi_N\mathcal{L}_{\varepsilon}.
$$
According to the Trevisan result (for the case of Dirac's initial condition,
see also \cite[Theorem~3.2.6]{StrV}),
there  exists a solution $Q_{\varepsilon}^N$ to the martingale problem
associated with~$\mathcal{L}_\varepsilon^N$ and the initial condition $\sigma^{\varepsilon}(0, x)\,dx$.
Let $\varrho^N_t(dx)dt$ be the corresponding probability
solution to the Cauchy problem with $\mathcal{L}_\varepsilon^N$.

As in the proof of  \cite[Theorem 2.5]{ManSh} and \cite[Theorem 6.7.3]{book},
 one can choose a subsequence
$\{N_k\}$ such that $\varrho^{N_k}_t(dx)\,dt$ converges weakly to $\varrho_t(dx)\,dt$
on every compact set in $[0, T]\times\mathbb{R}^d$
and $\varrho^{N_k}_t$ converges weakly to $\varrho_t$ on every compact set in $\mathbb{R}^d$
for each $t\in[0, T]$. Moreover, $\varrho_t\,dt$ is a sub-probability solution to the
Cauchy problem with $\mathcal{L}_{\varepsilon}$ and initial condition $\sigma^{\varepsilon}(0, x)\,dx$.
By the cited uniqueness result for the Cauchy problem
we have $\varrho_t(dx)=\sigma^{\varepsilon}(t, x)\,dx$.

Let us prove that the family of measures $Q_{\varepsilon}^{N_k}$ is compact in the weak topology.
Let $q>1$ and $\zeta_{q}(t)=t$ if $t<q-1$
and $\zeta_{q}(t)=q$ if $t>q+1$, $0\le \zeta_q'(t)\le 1$, $-c\ge\zeta_q''\le 0$,
where $c$ does not depend on $q$.
Set $V_1=\theta(V)$. Note that
$$
Q_{\varepsilon}^{N_k} \Bigl(\omega\colon \sup_{t\in[0, T]}V_1(\omega(t))\ge q-1\Bigr)=
Q_{\varepsilon}^{N_k} \Bigl(\omega\colon \sup_{t\in[0, T]}\zeta_q(V_1(\omega(t)))\ge q-1\Bigr)
$$
Repeating the arguments from Proposition \ref{pt-ball}
with $\zeta_q(V_1)$ in place of $V$ and taking into account
that $\zeta_q'(V_1)=\zeta_q''(V_1)=0$ if $V_1>q+1$
we obtain the estimate
\begin{multline*}
Q_{\varepsilon}^{N_k} \Bigl(\omega\colon \sup_{t\in[0, T]}V_1(\omega(t))\ge q-1\Bigr)
\\
\le \frac{2C_7}{(q-1)}\Bigl(\int V_1(x)\sigma^{\varepsilon}(0, x)\,dx+
\int_0^T\int_{V_1\le q+1}\Bigl(|\mathcal{L}_{\varepsilon}V_1|
+|\sqrt{\mathcal{A}_{\varepsilon}}\nabla V_1|^2\Bigr)\varrho^{N_k}_t\,dx\,dt\Bigr).
\end{multline*}
Since $\{x\colon V_1(x)\le q+1\}$ is a compact set, for sufficiently large $N_k$
the last integral is close to
$$
\int_0^T\int_{V_1\le q+1}\Bigl(|\mathcal{L}_{\varepsilon}V_1|
+|\sqrt{\mathcal{A}_{\varepsilon}}\nabla V_1|^2\Bigr)\sigma^{\varepsilon}(t, x)\,dx\,dt.
$$
Thus, for every $\lambda\in(0, 1)$ one can take $q$ so large  that there exists a number $k_0$ such
that for every $k>k_0$ we have
$$
Q_{\varepsilon}^{N_k} \Bigl(\omega\colon \sup_{t\in[0, T]}V_1(\omega(t))\ge q-1\Bigr)\le \lambda.
$$
It follows that for every $\lambda\in(0, 1)$ there exists $R>0$ such that
$$
Q_{\varepsilon}^{N_k}\Bigl(\omega\colon \|\omega(t)\|\le R\Bigr)\ge 1-\lambda \quad \forall N_k.
$$
Repeating the arguments from the first part of the proof and
using the functions $\Theta_1$ and $\Theta_2$ appearing from Trevisan's result
and de la Vall\'ee Poussin's theorem for $\|A(t,x)\|I_{|x|\le R}$ and $|b(t,x)|I_{|x|\le R}$,
we obtain the estimate
\begin{multline*}
\mathbb{E}_{Q_{\varepsilon}^{N_k}}\Psi_d(f_i)\le
C_8\sup_x|\Theta(x_i\psi_R(x))|
\\
+C_8T\sup_{x, t}\Bigl(\Theta_1(|\mathcal{L}_{\varepsilon}(x_i\psi_R(x))|)
+\Theta_2(\|\mathcal{A}_{\varepsilon}(t, x)\|)\Bigr)
\end{multline*}
with the functions
 $\Theta(t)=\Theta_1(t)=\Theta_2(t)=t^2$. Here the right-hand side does not depend on $N_k$.
We have
$$
\mathbb{E}_{Q_{\varepsilon}} (I_{\|\omega\|\le R}\Psi_d )\le C_9(R),
$$
where $C_9(R)$ depends on $R$, but does not depend on $\varepsilon$.
For any number $\lambda\in(0, 1)$ one can take a sufficiently
large number $M$ such that for the compact set $K=\{\omega\colon \Psi_d(\omega)\le M, \ \|\omega\|\le R\}$
there holds the estimate
$$
Q_{\varepsilon}^{N_k}(K)\ge 1-2\lambda \quad \forall N_k.
$$
Therefore, the family of measures $Q_{\varepsilon}^{N_k}$ contains a sequence
$\{Q_{\varepsilon}^{N_{k_j}}\}$
weakly converging to some probability measure $Q_{\varepsilon}$. Let us verify that $Q_{\varepsilon}$ corresponds to the
solution $\sigma^{\varepsilon}$. Clearly, conditions (i) and (iii) are fulfilled.
Note that if $f\in C_0^{\infty}(\mathbb{R}^d)$,
 then $\varphi_{N}\mathcal{L}_{\varepsilon}f=\mathcal{L}_{\varepsilon}f$
for all sufficiently large $N$. Hence (ii) follows
from weak convergence of $Q_{\varepsilon}^{N_{k_j}}$.

Thus, for every $\varepsilon$
there  exists a probability measure $Q_{\varepsilon}$
on $\Omega_d$  for which (i), (ii),
(iii) are fulfilled with the operator $\mathcal{L}_{\varepsilon}$
and  $\sigma^{\varepsilon}$ solves the Cauchy problem. Moreover, by Proposition \ref{pt-ball}
for every $\lambda\in(0, 1)$ there exists $R>0$ such that
$$
Q_{\varepsilon}\Bigl(\omega\colon \|\omega\|\le R\Bigr)\ge 1-\lambda \quad
\forall \varepsilon>0.
$$

\vskip .1in
{\bf IV. Verification of the compactness of the family of measures $Q_{\varepsilon}$ and
the proof of the fact  that the limit is the  required measure.}

We need the following version of Jensen's inequality.
Let $\Phi$ be a convex and increasing function on $[0, +\infty)$ with $\Phi(0)=0$,
$\nu$ a subprobability measure on some space $X$, and $f\ge 0$ a measurable function on~$X$.
Then
$$
\Phi\biggl(\int_X f\,d\nu\biggr)\le \int_X \Phi(f)\,d\nu,
$$
which follows by Jensen's  inequality  and the
inequality $\Phi(\alpha t)\le\alpha\Phi(t)$ for $\alpha\in [0,1]$
that  follows from the convexity of $\Phi$.
Let
$\chi_R\ge 0$ be a smooth function such that
$\chi_R(x)=1$ if $|x|\le 2R$ and $\chi_R(x)=0$ if $|x|\ge 3R$.
Note that
$$
\frac{\varepsilon\gamma(x)}{\sigma^{\varepsilon}(t, x)}\le 1.
$$
Hence
$$
\Phi\Bigl(|\beta_{\varepsilon}(x, t)|\chi_R(x)+\frac{\varepsilon\gamma(x)|x|}{\sigma^{\varepsilon}(t, x)}\chi_R(x)\Bigr)\le
\frac{1}{2}\Phi(2|\beta_{\varepsilon}(x, t)|\chi_R(x))+\frac{1}{2}\Phi(6R).
$$
Since
$$
\frac{1-\varepsilon}{\sigma^{\varepsilon}(t, x)}\int
\int h_{\varepsilon}(t-s, x-y)\, \mu_s^{\delta}(dy)\,ds\le 1,
$$
one has
\begin{multline*}
\Phi(2|\beta_{\varepsilon}(x, t)|\chi_R(x))=
\Phi\biggl(\frac{1-\varepsilon}{\sigma^{\varepsilon}(t, x)}
\int \int 2|b_{\delta}(s, y)|\chi_R(x)h_{\varepsilon}(t-s, x-y)\, \mu_s^{\delta}(dy)\,ds
\biggr)
\\
\le \frac{1-\varepsilon}{\sigma^{\varepsilon}(t, x)}
\int \int\Phi(2|b_{\delta}(s, y)|\chi_R(x))h_{\varepsilon}(t-s, x-y)\, \mu_s^{\delta}(dy)\,ds.
\end{multline*}
Therefore, we obtain
$$
2\int_0^T \int\Phi\Bigl(|\beta_{\varepsilon}|\chi_R+\frac{\varepsilon\gamma(x)|x|}{\sigma^{\varepsilon}(t, x)}\chi_R\Bigr)
\, \sigma^{\varepsilon}(x, t)\,dx\,dt\le
\int_0^{T_1}\int\Phi(2|b|\chi_R)\,d\mu_t\,dt+T_1\Phi(6R).
$$
In the same way we obtain
$$
2\int_0^T\int\Phi(\|\mathcal{A}_{\varepsilon}\|\chi_R)
\, \sigma^{\varepsilon}(x, t)\,dx\,dt\le
\int_0^{T_1}\int\Phi(2\|A\|\chi_R)\,d\mu_t\,dt+\Phi(2).
$$
Analogous estimates are fulfilled for
$$\|\mathcal{A}_{\varepsilon}\|+|\beta_{\varepsilon}|+\frac{\varepsilon\gamma(x)|x|}{\sigma^{\varepsilon}(t, x)}.$$
Repeating the arguments from the first part of the proof
and again
using the functions $\Theta_1$ and $\Theta_2$ appearing from Trevisan's result
and de la Vall\'ee Poussin's theorem for $\|A(t,x)\|I_{|x|\le R}$ and $|b(t,x)|I_{|x|\le R}$,
we obtain the estimate
\begin{multline*}
\mathbb{E}_{Q_{\varepsilon}}\Psi_d(f_i)\le \int\Theta(x_i\psi_R(x))
\,\mu_{\delta}(dx)
\\
+
\int\int \Bigl(\Theta_1(|\mathcal{L}_{\varepsilon}(x_i\psi_R(x))|)+\Theta_2
(|\sqrt{\mathcal{A}_{\varepsilon}}\nabla(x_i\psi_R(x))|^2)\Bigr)\, \sigma^{\varepsilon}(x, t)\,dx\,dt,
\end{multline*}
where $f_i(x)=x_i\psi_R(x)$ and the right-hand side is estimated by
\begin{multline*}
\sup_x|\Theta(x_i\psi_R(x))|
\\
+C_{10}\int_0^{T_1}\int_{|x|\le 2R}\Bigl(\Theta_1
(\|A(t, x)\|+|b(t, x)|)+\Theta_2(\|A(t, x)\|)\Bigr)\, \mu_t(dx)\,dt
\\
+\Theta_1(6R)+\Theta_2(2),
\end{multline*}
which does not depend on $\varepsilon$.
Here we apply the above estimates with $\Phi=\Theta_1$ and $\Phi=\Theta_2$ for the measure~$Q_\varepsilon$.
We have
$$
\mathbb{E}_{Q_{\varepsilon}} (I_{\|\omega\|\le R}\Psi_d )\le C_{11}(R)
$$
where $C_{11}(R)$ depends on $R$, but does not depend on $\varepsilon$.
For any number $\lambda\in(0, 1)$ one can take a sufficiently
large number $M$ such that for the compact set $K=\{\omega\colon \Psi_d(\omega)
\le M, \ \|\omega\|\le R\}$
there holds the estimate
$$
Q_{\varepsilon}(K)\ge 1-2\lambda \quad \forall \varepsilon>0.
$$
Therefore, the family of measures $Q_{\varepsilon}$ contains a sequence $\{Q_{\varepsilon_k}\}$ with
$\varepsilon_{k}\to 0$ weakly converging to some probability measure $P_{\delta}$.

Let us verify that the measure $P_{\delta}$ satisfies (i), (ii) and (iii).
The first and last properties are obtained in the limit letting $\varepsilon_k\to 0$
in the equality
$$
\int_{\Omega_d} f(\omega(t))\, Q_{\varepsilon_k}(d\omega)
=\int f(x)\sigma^{\varepsilon_k}(x, t)\,dx
 \quad \forall f\in C_0^{\infty}(\mathbb{R}^d).
$$
For the proof of~(ii) (the martingale property), as above,
 we have to show that for every bounded continuous
function
$g\colon \Omega_d\to\mathbb{R}$
that is measurable with respect to the $\sigma$-algebra $\mathcal{F}_s$ there holds the equality
$$
\int_{\Omega_d}\Bigl[f(\omega(t))-f(\omega(s))-\int_s^t
L_{\delta}f(\tau,\omega(\tau))\, d\tau\Bigr]g(\omega)\, P_{\delta}(d\omega)=0, \quad t\ge s,
$$
where $f\in C_0^{\infty}(\mathbb{R}^d)$.
To  this end, exactly as at the first step,
it suffices to show that
$$
\lim_{\varepsilon_k\to 0}\int_{\Omega_d}\biggl[\int_s^t
\mathcal{L}_{\varepsilon_k}f(\tau, \omega(\tau))\, d\tau\biggr]g(\omega)\,
Q_{\varepsilon_k}(d\omega)
=
\int_{\Omega_d}\biggl[\int_s^t
L_{\delta}f(\tau,\omega(\tau))\, d\tau\biggr]g(\omega)\,P_{\delta}(d\omega).
$$
Let $q^{ij}, z^i\in C^{\infty}([-1, T_1]\times\mathbb{R}^d)$ and
$\widetilde{L}=q^{ij}\partial_{x_i}\partial_{x_j}+z^i\partial_{x_i}$.
It is clear that the difference
$$
\int_{\Omega_d} \biggl[\int_s^t \widetilde{L}f(\tau,\omega(\tau))\, d\tau
\biggr]g(\omega)\,Q_{\varepsilon_k}(d\omega)
-
\int_{\Omega_d}\biggl[\int_s^t \widetilde{L}f(\tau,\omega(\tau))\, d\tau
\biggr]g(\omega)\,P_{\delta}(d\omega)
$$
tends to zero again by Lemma~\ref{lem1}. Moreover, the expression
$$
\biggl|\int_{\Omega_d}\biggl[\int_s^t (\widetilde{L}-L_{\delta})
f(\tau,\omega(\tau))\biggr]g(\omega)\,P_{\delta}(d\omega)\biggr|
$$
is estimated by
$$
\int_0^T\int_{\mathbb{R}^d} \Bigl(|a^{ij}_{\delta}-q^{ij}|\, |\partial_{x_i}\partial_{x_j}f|
+|b^i_{\delta}-z^i|\, |\partial_{x_i}f|\Bigr)\,d\mu_t^{\delta}\,dt,
$$
which can be made arbitrarily small by a suitable choice of $q^{ij}$ and $z^i$
approximating $a^{ij}_{\delta}$ and $b^i_{\delta}$ in $L^1$ with respect to the measure
$\mu_t^{\delta}\, dt$ on $[0, T_1]\times U$, where $U$ is a ball containing the support of~$f$.
Set
$$
z_{\varepsilon}^i(t, x)
=\frac{1-\varepsilon}{\sigma^{\varepsilon}(t, x)}
\int\int z^i(s, y)h_{\varepsilon}(t-s, x-y)\, \mu_s^{\delta}(dy)\,ds,
$$
$$
q^{ij}_{\varepsilon}(t, x)
=\frac{1-\varepsilon}{\sigma^{\varepsilon}(t, x)}
\int\int q^{ij}(s, y)h_{\varepsilon}(t-s, x-y)
\, \mu_s^{\delta}(dy)\,ds,
$$
$$
\widetilde{L}_{\varepsilon}=q^{ij}_{\varepsilon}\partial_{x_i}
\partial_{x_j}+z_{\varepsilon}^i\partial_{x_i}.
$$
Note that $z^i_{\varepsilon}-z^i$ and $q_{\varepsilon}^{ij}-q^{ij}$ converge to zero
uniformly on $[0, T_1]\times U$
for every ball~$U$. Since the functions $(\widetilde{L}-\widetilde{L}_{\varepsilon_k})f(t,x)$ converge to zero uniformly on
$[0, T]\times\mathbb{R}^d$ as $\varepsilon_k\to 0$, the expression
$$
\int_{\Omega_d}\biggl[\int_s^t (\widetilde{L}
-\widetilde{L}_{\varepsilon_k})f(\tau,\omega(\tau))\, d\tau\biggr]g(\omega)\,Q_{\varepsilon_k}(d\omega)
$$
tends to zero as $\varepsilon_k\to 0$.
Let
$$
C(f)=
d\sup_x|\nabla f(x)|+d\sup_x\|D^2f(x)\|.
$$
and $f(x)=0$ if $|x|>r$.
Note that
\begin{multline*}
|(\widetilde{L}_{\varepsilon_k}
-\mathcal{L}_{\varepsilon_k})f(t, x)|
\le \frac{1-\varepsilon_k}{\sigma^{\varepsilon_k}(t, x)}
\int\int \bigl[|a^{ij}_{\delta}(s, y)-q^{ij}(s, y)||\partial_{x_i}\partial_{x_j}f(x)|
\\
+|b^i_{\delta}(s, y)-z^i(s, y)||\partial_{x_i}f(x)|\bigr]
h_{\varepsilon_k}(t-s, x-y)\, \mu_s^{\delta}(dy)\,ds
\\
+\frac{\varepsilon_k\gamma}{\sigma^{\varepsilon_k}}(|\Delta f|+|x||\nabla f|)
\\
\le C(f)\frac{1-\varepsilon_k}{\sigma^{\varepsilon_k}(t, x)}
\int\int_{|y|\le r+1}\bigl[|a^{ij}_{\delta}(s, y)-q^{ij}(s, y)|
+|b^i_{\delta}(s, y)-z^i(s, y)|\bigr]
\\
\times h_{\varepsilon_k}(t-s, x-y)\, \mu_s^{\delta}(dy)\,ds\,dx +
\varepsilon_k C(f)\frac{(1+|x|)\gamma(x)}{\sigma^{\varepsilon_k}}.
\end{multline*}
Therefore, the integral
$$
\int_{\Omega_d}\biggl[\int_s^t (\widetilde{L}_{\varepsilon_k}
-\mathcal{L}_{\varepsilon_k})f(\tau,\omega(\tau))\, d\tau\biggr]g(\omega)\,Q_{\varepsilon_k}(d\omega)
$$
is estimated by
\begin{equation}\label{b1}
C(f)\int_0^{T}\int_{|x|\le r+1}\Bigl(|a^{ij}_{\delta}-q^{ij}|+|b^i_{\delta}-z^i|\Bigr)\,d\mu_t^{\delta}\,dt
+\varepsilon_k C(f)\int (1+|x|)\gamma(x)\,dx,
\end{equation}
which can be made arbitrarily small by a suitable choice of $q^{ij}$ and $z^i$ as above.
Denoting by $I(L, P)$ the integral
$$
\int_{\Omega_d}\biggl[\int_s^t
Lf(\tau, \omega(\tau))\, d\tau\biggr]g(\omega)\,P(d\omega),
$$
we have
\begin{multline*}
I(\mathcal{L}_{\varepsilon_k}, P_{\varepsilon_k})-I(L_{\delta}, P_{\delta})=
(I(\mathcal{L}_{\varepsilon_k}, P_{\varepsilon_k})-I(\widetilde{L}_{\varepsilon_k}, P_{\varepsilon_k}))
+(I(\widetilde{L}_{\varepsilon_k}, P_{\varepsilon_k})-I(\widetilde{L}, P_{\varepsilon_k}))
\\
+(I(\widetilde{L}, P_{\varepsilon_k})-I(\widetilde{L}, P_{\delta}))
+(I(\widetilde{L}, P_{\delta})-I(L_{\delta}, P_{\delta})).
\end{multline*}
Let $\lambda>0$. First we take $q^{ij}$ and $z^i$ such that for
the first and  forth terms in
the right-hand side the integrals with $\mu_t^\delta$ in
the corresponding bounding expressions (\ref{b1}) are smaller  than~$\lambda$.
Next we take $k$ such that the part with $\varepsilon_k$ in (\ref{b1})
for the first term and also the second and third terms are smaller than~$\lambda$.
It follows that $I(\mathcal{L}_{\varepsilon_k}, P_{\varepsilon_k})-I(L_{\delta}, P_{\delta})\to 0$.
Thus, we have verified (i), (ii), (iii) for~$P_{\delta}$.

\vskip .1in
{\bf V. Extension to the whole interval.}

We have constructed martingale representations $P_n$ defined
on $C([0,T-1/n],\mathbb{R}^d)$ for every smaller interval $[0,T-1/n]$.
It now remains to observe that Trevisan's a priori estimate employed above
(that is, \cite[Theorem~A2 and Corollary~A5]{TREV})
enables one to construct a representation on the whole interval~$[0,T]$ on which we have a
solution to the Fokker--Planck--Kolmogorov equation.
To this end we extend the measures $P_n$ to $C([0,T],\mathbb{R}^d)$ by using
the natural extension operator that associates to every function $\omega$
on $[0,T-1/n]$ the function that extends it by the constant value $\omega(T-1/n)$ on $(T-1/n,T]$.
In addition, the compact function $\Psi_d$ on $C[0,T-1/n],\mathbb{R}^d)$
ensured by Trevisan's result for each~$P_n$ (used at Step~I) can be chosen in a unified way, namely,
by taking such a function on $C([0,T],\mathbb{R}^d)$ and then restricting it
to $C([0,T-1/n],\mathbb{R}^d)$ embedded into $C([0,T],\mathbb{R}^d)$ by means of extensions
as explained above. It is readily seen from the formulation of the cited results from \cite{TREV}
mentioned above that in this way we obtain a compact function on $\Omega_d$ the integrals
of which with respect to the extensions of $P_n$ to $\Omega_d$ remain uniformly bounded
(here it is important, of course, that in our condition~(\ref{con-new})
the integral is taken over all of~$[0,T]$).
Hence the sequence of extensions of measures $P_n$ contains a weakly convergent
subsequence. The limit of this subsequence gives the desired representation.
The verification of this is analogous to the previous steps. Of course, the main
point is check the martingale property, which is not automatic in case
of discontinuous $A$ and $b$, but follows again my smooth approximations and estimate~(\ref{ek1}).
\end{proof}

The main difficulty with the smoothing of coefficients is due to the necessity to obtain in
the limit the solution we consider (but not an arbitrary solution, since there can exist many),
in addition, for the approximating solutions we have to keep our Lyapunov-type condition.

\begin{remark}
\rm
If the Cauchy problem has a solution on the whole half-line, then a similar reasoning
gives a representing martingale measure on the space of paths on $[0,+\infty)$,
but here one must be careful, since this space is not separable, so the desired measure
is defined not on all Borel sets, but on some smaller $\sigma$-field.
\end{remark}

Finally, we give an example showing that the integrability of $(1+|x|)^{-2}|\langle b(t, x), x\rangle|$
can hold even in the case where the
function $(1+|x|)^{-1}|b(t, x)|$ is not integrable  with respect to the solution.

\begin{example}\label{ex2}
\rm
Let $d=2$ and let $(r, \varphi)$ be polar coordinates.
We construct an example of a stationary solution $\mu=\varrho\,dx$
to the equation with the unit matrix  $A$ and the drift coefficient
$b=\nabla\varrho/\varrho$ for a suitable function~$\varrho$.
We recall that
$$
|\nabla\varrho(x)|^2=r^{-2}|\partial_{\varphi}\varrho|^2
+|\partial_r\varrho|^2.
$$
Therefore, it suffices to find a  smooth nonnegative
function $\varrho$ for which
$$
\varrho, (1+r)^{-1}|\partial_r\varrho|\in L^1(\mathbb{R}^2),
\quad (1+r)^{-2}|\partial_{\varphi}\varrho|\notin L^1(\mathbb{R}^2).
$$
Set
$$
\varrho(r, \varphi)=\sum_{n=1}^{\infty}2^{-n}\psi(r-n)(2
+\sin(4^n\varphi)),
$$
where $\psi\in C^{\infty}_0((0, 1))$ and $\psi\ge 0$ is not identically zero.
We observe that for every point $(r, \varphi)\in(0, +\infty)\times[0, 2\pi]$
only one term of the series is nonzero.
It is clear that $\varrho\in C^{\infty}(\mathbb{R}^2)$, $\varrho\ge 0$ and
$$
\varrho(r, \varphi)+|\partial_r\varrho(r, \varphi)|\le C2^{-r}
$$
for some number $C>0$. However,
$$
|\partial_{\varphi}\varrho(r, \varphi)|=\sum_{n=1}^{\infty}2^{n}\psi(r-n)|
\cos(4^n\varphi)|.
$$
Since the integral of $|\cos(4^n\varphi)|$ is estimated from below by $2\pi$, we have
$$
\int_0^{\infty}\int_0^{2\pi}r^{-1}|\partial_{\varphi}\varrho|
\,d\varphi\,dr\ge
\sum_{n=1}^{\infty}2\pi c_{\psi}(n+1)^{-1}2^{n}=+\infty, \quad c_{\psi}
=\int_0^1|\psi(s)|\,ds.
$$
Thus, $r^{-2}|\partial_{\varphi}\varrho|\notin L^1(\mathbb{R}^2)$.
In this example $b$ is a gradient. An example without this additional property
is even simpler. We observe that the standard Gaussian density
$\gamma$ is a stationary solution to the equation with $A=I$ and $b(x)=-x$,
so it remains a solution for the equation with a perturbed drift $-x+v(x)$,
where a smooth vector field $v$ is chosen such  that ${\rm div}\, (\gamma v)=0$
and $(x,v(x))=0$.  For example, we can take $v$ of the form
$v(x)=\gamma(x)^{-1}h(|x|^2)Ux$ with an orthogonal operator $U$ such that $(Ux,x)=0$. Of course,
$h$ can be rapidly increasing, so that $|v|\gamma$ will not be integrable.
\end{example}

\begin{remark}
\rm
The presented results can be extended with minor technical changes to as follows.
Let $V\in C^2([1, +\infty)$, there is $C>0$ such that
$$
\Bigl|\frac{V''(s)}{V''(t)}\Bigr|+\Bigl|\frac{V'(s)}{V'(t)}\Bigr|\le C
\quad \hbox{whenever $|t-s|\le 1$,}
$$
$V\ge 0$, $|V''|+|V'|\le C$ and $\lim\limits_{s\to+\infty}V(s)=+\infty$,
i.e., the integral of $V'$ diverges.
Then condition (\ref{con-new}) can be replaced by
\begin{multline*}
\int_0^T\int_{\mathbb{R}^d}
\Bigl[ \Bigl(|V''(1+|x|^2)|(1+|x|^2)+|V'(1+|x|^2)|\Bigr)\|A(t,x)\|
\\
 +|\langle b(t,x), x\rangle|\, |V'(1+|x|^2)|\Bigr]\, \mu_t(dx)\,dt<\infty.
\end{multline*}
For $V(s)=\log s$ we obtain the original condition (\ref{con-new}).
If $V(s)=\log(1+\log s)$, then we arrive at the condition
$$
\frac{\|A(t,x)\|}{(1+|x|^2)\log(1+|x|^2)}, \
\frac{|\langle b(t,x), x\rangle|}{(1+|x|^2)\log(1+|x|^2)}\in L^1(\mu_t\,dt).
$$
\end{remark}

\begin{remark}\label{re3.4}
\rm
The superposition principle applies not only to linear Fokker--Planck--Kolmogorov equations,
but also to nonlinear equations. Let $\{\mu_t\}$ be a solution to the Cauchy problem
$$
\partial_t\mu_t=\partial_{x_i}\partial_{x_j}(a^{ij}(t, x, \mu)\mu_t)
-\partial_{x_i}(b^i(t, x, \mu)\mu_t), \quad \mu_0=\nu.
$$
For a precise definition of a solution and typical examples of dependence of $A$ and $b$ on
the solution $\mu$ are given in \cite[Chapter~6]{book}, \cite{ManRSh15}, \cite{ManSh14}.
In particular, typical global assumptions are expressed in terms of a Lyapunov function~$V$:
$$
L_{\mu}V\le C(\mu)+C(\mu)V, \quad V\in L^1(\nu).
$$
If $V(x)=\log(1+|x|^2)$, then by Proposition~\ref{pr1} the solution
$\{\mu_t\}$ satisfies condition~(\ref{con-new}).
Given a solution  $\{\mu_t\}$, we can regard it as a solution to the
linear operator~$L_{\mu}$. Therefore, there exists the corresponding solution $P_{\nu}$ to the
martingale problem
such that $\mu_t$ is the one-dimensional distribution of the measure $P_{\nu}$ on~$C([0, T], R^d)$.
Hence we can assume that the measure $P_{\nu}$ solves the martingale problem with the operator $L_{\mu}$
that depends on $P_{\nu}$ through~$\mu$,
i.e., solves the martingale problem corresponding to the stochastic McKean--Vlasov equation
(see \cite{Funaki}).
Thus, using the superposition principle and solutions to the Fokker--Planck--Kolmogorov
equation one can construct solutions to the martingale problem for nonlinear
stochastic equations. This approach is applied for constructing probabilistic
representations of solutions to PDEs (see, e.g., \cite{BarbuR}).
\end{remark}

It is worth noting that the superposition principle can be useful for the study
of uniqueness problems for Fokker--Planck--Kolmogorov equations with coefficients
of low regularity, on this topic see the book \cite{book} and the papers
 \cite{BRS11a}, \cite{BRS11b}, \cite{BRS13}, \cite{BRS15}, \cite{LL}, \cite{Sh12}, \cite{Z}.
We also plan to study analogous questions for infinite-dimensional equations
in the spirit of~\cite{BDPR10}, \cite{BDPR11} and~\cite{BDPRS15}.

This research was supported by the RFBR Grants 18-31-20008 and 17-01-00662,
the CRC 1283 at Bielefeld University and the DFG Grant  DFG RO 1195/12-1.

\end{document}